\documentclass[11pt]{article}
\usepackage{layout,verbatim,amssymb,amsmath,setspace,cite}
\usepackage{amsthm,authblk}      
\usepackage{graphicx}
\usepackage[dvipsnames]{xcolor}
\usepackage{color,latexsym,amsmath,amsfonts,amssymb}
\usepackage[utf8]{inputenc}
\usepackage{layout,verbatim,amssymb,amsmath,setspace,cite}

\newcommand{\vG}{\varGamma}
\newcommand{\ve}{\varepsilon}

\newcommand{\ov}{\overline}

\def\N{{\mathbb{N}}}
\def\R{{\mathbb{R}}}

\def\mcL{{\mathcal L}}

\def\mcI{{\mathcal I}}

\def\be{\begin{equation}}
\def\ee{\end{equation}}
\def\ba*{\begin{eqnarray*}}
\def\ea*{\end{eqnarray*}}

\def\vPh{\varPhi}
\providecommand{\msc}[1]{\textbf{{M.S.C.:}} #1}

\newtheorem{thm}{Theorem}[section]{\bfseries}{\itshape}
\newtheorem{cor}[thm]{Corollary}{\bfseries}{\itshape}
\newtheorem{lem}[thm]{Lemma}{\bfseries}{\itshape}
{\bfseries}{\itshape}
\newtheorem{prop}[thm]{Proposition}{\bfseries}{\itshape}
\newtheorem{deff}[thm]{Definition}{\bfseries}{\itshape}%
{\bfseries}{\itshape}
\newtheorem{rem}[thm]{Remark}{\bfseries}{\itshape}
\newtheorem{ex}[thm]{Example}{\bfseries}{\itshape}

\begin{document}

\title{Some new results on integration for multifunction\thanks{This research was supported by
the Grant  Prot. N.  U2015/001379 of GNAMPA -- INDAM (Italy);
by University of Perugia -- Dept. of Mathematics and Computer Sciences -- Grant Nr 2010.011.0403
and by University of Palermo (Italy).}
}

\author[Candeloro, Di Piazza, Musia{\l}, Sambucini]{Domenico Candeloro, Luisa Di Piazza, Kazimierz Musia{\l}, Anna Rita Sambucini }
\date{}

\maketitle

\begin{abstract}
It has been proven in \cite{dm2} and \cite{dm4} that each Henstock-Kurzweil-Pettis integrable multifunction
with weakly compact values can be represented as a sum of one of its selections  and a Pettis integrable multifunction.
We prove here that if the initial multifunction is also Bochner measurable and has absolutely continuous variational
measure of its integral, then it is a sum of a strongly measurable selection and of a variationally Henstock integrable multifunction
that is also Birkhoff integrable (Theorem \ref{BOH}).
\end{abstract}
{\bf keywords} Multifunction, set-valued Pettis integral, set-valued variationally Henstock and Birkhoff integrals, selection\\
\msc{   28B20 ,  26E25 , 26A39 , 28B05 , 46G10 , 54C60 , 54C65}

\section{Introduction}\label{intro}

Integration of vector valued functions is strongly motivated by general
problems of modern analysis including control theory, economics and differential inclusions.
In many situations it is enough to use the
well-known Bochner and Pettis  integrals involved with the classical Lebesgue theory.
In the last decades however gauge (non--absolute) integrals have been also considered
\cite{BCS,BS2011,cdpms2016, lms, ncm, CASCALES2, crbb, DiP,dm11, dm, f1994a, fm, lm,  mpot}
after the pioneering studies of  G. Birkhoff, J. Kurzweil, R. Henstock and  E. J. McShane.
The Birkhoff integral was introduced in \cite{birkhoff} and recently
investigated in \cite{ccgs,ckr1, ckr2009, ma, mpot}.
 B.  Cascales, V. M. Kadets, M. Potyrala, J. Rodriguez,  and other authors considered the unconditional
Riemann--Lebesgue multivalued integral.\\
This article is organized in the following manner. In Section \ref{two} we give  some preliminaries
and the definitions.
In Section \ref{three} a new decomposition theorem for Bochner measurable and weakly compact valued
 Henstock-Kurzweil-Pettis integrable multifunction $\Gamma$ with absolutely continuous variational
measure of its integral is obtained, involving Birkhoff and variationally Henstock integrability of the multifunction $G$ such that $\Gamma = f + G$.
It is also shown that the conditions in Theorem \ref{BOH} do not imply, in general, the variational integrability of the multifunction $\Gamma$ and an example is given (Example \ref{exc0}). \\
Finally, in case of strongly  measurable (multi)functions,
a characterization of the Birkhoff integrability is given using a kind of Birkhoff strong property (see Definition \ref{sBi}).

\section{Preliminary facts}\label{two}
Throughout $[0,1]$ is the unit interval of the real line equipped
with the usual topology and  Lebesgue measure $\lambda$,
$\mathcal{L}$ denotes the family of all Lebesgue measurable subsets of $[0,1]$, and
 $\mcI$
is the collection of all closed subintervals of $[0,1]$: if $I\in  \mcI$
then its Lebesgue  measure will be denoted by $|I|$.

$X$ is an arbitrary  Banach space with its dual $X{}^*$.
  The closed unit ball of $X{}^*$ is denoted by
  $B(X{}^*)$. $cwk(X)$ is    the family of all non-empty  convex weakly compact subsets of $X$ and  $ck(X)$ is
 the family of all  compact members of $cwk(X)$ .  We consider on $cwk(X)$ the
   Minkowski addition ($A+B :\,=\{a+b:a\in  A,\,b\in  B\}$) and the
   standard multiplication by scalars.
$d{}_H$ is the Hausdorff distance in $cwk(X)$ and $cwk(X)$ with this metric is
   a complete metric space.

For every $A\in cwk(X)$,   $\|A\|:=d_H(A,\{0\})=\sup\{\|x\|\colon x\in {A}\}$.\\
     For every $C \in  cwk(X)$ the {\it support
  function of} $C$ is denoted by $s( \cdot, C)$ and defined on $X{}^*$ by $s(x{}^*, C) = \sup \{ \langle x{}^*,x \rangle : \ x
  \in  C\}$, for each $x{}^* \in  X{}^*$. \\
  A map $\vG:[0,1]\to 2{}^X\setminus\{\emptyset\}$
  (= non-empty subsets of $X$) is called a {\it multifunction}.\\
 A multifunction $\vG:[0,1]\to cwk(X)$ is said to be {\it scalarly measurable} if for every $ x{}^* \in  X{}^*$, the map
  $s(x{}^*,\vG(\cdot))$ is measurable.\\ $\vG$ is said to be {\it Bochner measurable}  if there exists a sequence of simple multifunctions $\vG_n: [0,1] \to cwk(X)$ such that
  $$\lim_{n\rightarrow \infty}d_H(\vG_n(t),\vG(t))=0$$ for almost all $t \in [0,1]$.\\
  A function $f:[0,1]\to X$ is called a {\it selection of} $\vG$ if $f(t) \in \vG(t)$,   for every $t\in  [0,1]$.
\\
 A {\it partition} ${\mathcal P}$ {\it in} $[0,1]$ is a collection $\{(I{}_1,t{}_1),$ $ \dots,(I{}_p,t{}_p) \}$,
  where $I{}_1,\dots,I{}_p$ are nonoverlapping subintervals of $[0,1]$, $t{}_i$ is a point of $[0,1]$, $i=1,\dots, p$.\\
 If $\cup^p_{i=1} I{}_i=[0,1]$, then  ${\mathcal P}$ is {\it a partition of} $[0,1]$. If   $t_i$ is a point of $I{}_i$, $i=1,\dots,p$,  we say that ${\mathcal P}$ is a {\it Perron partition of}
  $[0,1]$.\\
 A {\it gauge} on $[0,1]$ is a positive function on $[0,1]$. For a given gauge $\delta$ on $[0,1]$,
  we say that a partition $\{(I{}_1,t{}_1), \dots,(I{}_p,t{}_p)\}$ is $\delta$-{\it fine} if
  $I{}_i\subset(t{}_i-\delta(t{}_i),t{}_i+\delta(t{}_i))$, $i=1,\dots,p$.\\

We recall that an interval multifunction $\Phi:{\mathcal I} \rightarrow cwk(X)$ is said to be \textit{finitely additive} if for every non-overlapping intervals $I{}_1, I{}_2 \in \mathcal{I}$ such that $I{}_1 \cup I{}_2 \in {\mathcal I}$ we have $\Phi(I{}_1 \cup I{}_2)=\Phi(I{}_1) + \Phi(I{}_2)$.\\
A multifunction $M:\mathcal{L} \rightarrow cwk(X)$ is said to be a $d_H$\textit{-multimeasure} if for every sequence $(A_n)_{n\geq 1}\subset \mathcal{L}$ of pairwise disjoint sets with $A=\bigcup_{n\geq 1}A_n$, we have $d_H(M(A),\sum_{k=1}^n M(A_k))\rightarrow 0$ as $n\rightarrow +\infty$.\\
 It is well known that if $M$ is $cwk(X)$-valued, then $M$ is a $d{}_H$-multimeasure if and only if it is a multimeasure, i.e.  if for every $x{}^*\in X{}^*$, the map
$A\mapsto s(x{}^*,M(A))$ is a real-valued measure (see \cite[Theorem 8.4.10]{hp}).\\
We say that the multimeasure $M:\mathcal{L} \rightarrow cwk(X)$ is $\lambda$\textit{-continuous} and we write $M\ll\lambda$, if $\lambda(A)=0$ yields $M(A)=\{0\}$.\\

    \begin{deff} \rm 
    A multifunction $\vG:[0,1]\to cwk(X)$ is said to be {\it Birkhoff}
   integrable on $[0,1]$,
   if there exists a non empty closed convex set  $\vPh{}_{\vG}([0,1]) \in cwk(X)$
 with the following property: for every $\varepsilon > 0$  there is a countable
partition $\Pi{}_0$  of $[0,1]$  in $\mathcal{L}$ such that for every countable partition $\Pi = (A{}_n){}_n$
 of $[0,1]$  in $\mathcal{L}$
finer than $\Pi_0$ and any choice $T = (t{}_n){}_n$  in $A{}_n$, the series
$\sum_n\lambda(A{}_n) \vG(t{}_n)$
 is unconditionally
convergent (in the sense of the Hausdorff metric) and
\begin{eqnarray}\label{e14-a}
d{}_H \left(\vPh{}_{\vG}([0,1]),\sum_n \vG(t{}_n) \lambda(A{}_n) \right)<\ve\,.
\end{eqnarray}
\end{deff}
    \begin{deff} \rm
 A multifunction $\vG:[0,1]\to cwk(X)$ is said to be {\it Henstock} (resp. {\it McShane})
   integrable on $[0,1]$,  if there exists a non empty closed convex set  $\vPh{}_{\vG}([0,1])\subset{X}$
    with the property that for every $\varepsilon > 0$ there exists a gauge $\delta$ on $[0,1]$
such that for each $\delta$--fine Perron partition (resp. partition)
   $\{(I{}_1,t{}_1), \dots,(I{}_p,t{}_p)\}$ of
   $[0,1]$, we have
\begin{eqnarray}\label{e14}
d{}_H \left(\vPh{}_{\vG}([0,1]),\sum_{i=1}^p\vG(t{}_i)|I{}_i|\right)<\ve\,.
\end{eqnarray}
A multifunction $\vG:[0,1]\to cwk(X)$ is said to be {\it Henstock} (resp.  McShane) integrable on
 $I\in \mcI$  (resp.  $E\in \mcL$ ) if
$\vG 1{}_I$
is Henstock (resp. $\vG 1{}_E$ is  McShane) integrable on $[0,1]$.
\end{deff}
If $X=\R$ and $\vG$ is a function, we speak of Henstock-Kurzweil integrability instead of Henstock one.
 Each of these integrals turns out to be an additive mapping, on $\mathcal{I}$ or $\mcL$ respectively, and is called {\em primitive}, or also {\em integral measure} of the mapping $\vG$.
From the definition and the completeness of the Hausdorff metric the $cwk(X)$-valued  integrals defined before have  weakly compact values.

 \begin{deff} \rm
 A multifunction $\vG:[0,1]\to cwk(X)$ is said to be {\it variationally Henstock} integrable on $[0,1]$,  if  there exists an additive mapping $\Phi_{\vG}:\mathcal{I}\to cwk(X)$
    with the property that for every $\varepsilon > 0$ there exists a gauge $\delta$ on $[0,1]$
such that for each $\delta$--fine Perron partition
   $\{(I{}_1,t{}_1), \dots,(I{}_p,t{}_p)\}$ of
   $[0,1]$, we have
\begin{eqnarray}\label{e14bis}
\sum_{i=1}^p d{}_H \left(\vPh{}_{\vG}(I_i),\vG(t{}_i)|I{}_i|\right)<\ve\,.
\end{eqnarray}
\end{deff}
 There is a large literature concerning
{\color{blue}} Pettis  and   Henstock-Kurzweil-Pettis  integral for functions and  multifunctions; we refer the
reader  to \cite{mu,mu3,mu4,mu8,ckr1,dm,dm2, pettis}.\\

A useful tool to study the $cwk(X)$-valued multifunctions is the R{\aa}dstr\"{o}m embedding (see for example \cite{l1}).
given by $i(A):=s(\cdot, A)$. It  satisfies the following properties:
\begin{itemize}
\item[1)] $i(\alpha A+ \beta C) = \alpha i(A) + \beta i(C)$ for every $A,C\in  cwk(X),\alpha, \beta \in  \mathbb{R}{}^+;$
\item[2)] $d_H(A,C)=\|i(A)-i(C)\|_{\infty},\quad A,C\in  cwk(X)$;
\item[3)] $i(cwk(X))=\ov{i(cwk(X))}$  (norm\ closure).
\end{itemize}

Observe that it follows directly from the definitions that
a multifunction $\vG:[0,1]\to{cwk(X)}$ is Birkoff (resp. Henstock,  McShane, variationally Henstock) integrable if and only if $i(\vG)$ is integrable in the same sense.

 \begin{deff}\rm \label{def2.4}
	A multifunction $\vG:[0,1]\to cwk(X)$ is said to be {\em Pettis} (resp.  {\em   Henstock-Kurzweil-Pettis integrable} (in short {\em(HKP)}-integrable)) if $s(x^*,\vG)$ is integrable (resp.
	Henstock-Kurzweil integrable)  for each $x^*\in X^*$, and for each $A \in \mathcal{L}$ (resp.  $I\in \mathcal{I}$)  there exists an element $w_A\in cwk(X)$ (resp. $w_I\in cwk(X)$) such that  $x^*(w_A)=\int_A s(x^*,\vG(t)) dt$ (resp. $x^*(w_I)=(H)\int_I s(x^*,\vG(t)) dt$) holds, for each $x^*\in X^*$. 
\end{deff}

\noindent
It follows from the classical properties of the Henstock-Kurzweil integral  that the primitives of Henstock or $(HKP)$ integrable multifunctions are interval multimeasures, while   the primitives of Pettis or Birkoff integrable multifunctions are multimeasures (see \cite[Theorem 4.1(i)]{ckr2009}).
\\

In particular we recall that, for a Pettis integrable mapping $G:[0,1]\to cwk(X)$, its integral $\varPhi_G$ is a countably additive multimeasure on the $\sigma$-algebra $\mathcal{L}$ (see \cite[Theorem 4.1]{ckr2009}) that is  absolutely continuous with respect to $\lambda$. As also observed in  \cite[Section 3]{ckr2009}, this means that the {\em embedded} measure $i(\varPhi_G)$ is a countably additive measure with values in  $l_{\infty}(B(X{}^*))$.
\\

For the definition of  the variational  measure $V_{\Phi}$ associated to a finitely additive interval measure $\Phi: {\mathcal I}\rightarrow \R$ we refer the reader to \cite{BDpM2, dpp}. In particular, we recall that the variational measure $V_{\phi}$ of the primitive $\phi$ of a variationally Henstock integrable mapping is a (possibly unbounded) $\lambda$-continuous measure on $\mcL$: see \cite[Proposition 3.3.1]{porcello}.

\section{Variationally Henstock integrable selections}\label{three}

Concerning  existence of selections of  gauge integrable mappings with values in $cwk(X)$, several results have been obtained recently  in
\cite{dm,dm2,BCS,cdpms2016,lms,mpot}.
The purpose of this paper  is to revisit the last results, trying to generalize it to a more general case.
We begin with some useful Lemmas.

\begin{lem}
Let $\vG:[0,1]\to cwk(X)$ be any $(HKP)$-integrable multifunction, and let $f:[0,1]\to X$ be any $(HKP)$-integrable selection of $\vG$.
Then, if $\phi$ and $\Phi$ denote respectively the primitives of $f$ and $\vG$, it is
\[\phi(I)\in \Phi(I)\]
for every subinterval $I\subset [0,1]$.
\end{lem}
\begin{proof}
It can be deduced as in the proof of \cite[Proposition 2.7]{cdpms2016}.
\end{proof}

\begin{lem}\label{misureselezioni}
Let $\vG:[0,1]\to cwk(X)$ be any $(HKP)$-integrable multifunction, and let $f$ be any $(HKP)$-integrable selection of $\vG$. If the variational measure $V_{\Phi}$ associated to the primitive  $\Phi$ of $\vG$ is absolutely continuous with respect to the Lebesgue measure $\lambda$, the same holds for the variational measure of the primitive $\phi$ of   $\,f$.
In particular this holds whenever $\vG$ is variationally Henstock integrable.
\end{lem}
\begin{proof}
From the previous Lemma we have that $\phi(I)\in \Phi(I)$ for every interval $I\subset [0,1]$. Then $V_{\phi}\leq V_{\Phi}$. Since  $V_{\Phi}$ is $\lambda$-continuous the same holds for $V_{\phi}$, of course.
\end{proof}
A crucial tool is the following.

\begin{thm}\label{ex41}
Let $\vG:[0,1]\to cwk(X)$ be any McShane
integrable mapping. Then $\vG$ is variationally Henstock integrable if and only if it is Bochner measurable and the variational measure of its integral is $\lambda$-continuous.
\end{thm}
\begin{proof} If $\vG$ is McShane integrable, then also $i\vG$ is McShane, hence Pettis, integrable and we may apply \cite[Lemma 4.1]{BDpM2}.
\end{proof}

Here is the new decomposition theorem.
\begin{thm}\label{BOH}
Let $\vG:[0,1]\to cwk(X)$ satisfy the following conditions:
\begin{description}
\item[\rm \ref{BOH}.1)]\ $\vG$ is Bochner measurable;
\item[\rm \ref{BOH}.2)]\  $\vG$ is $HKP$-integrable;
\item[\rm \ref{BOH}.3)]\  the variational measure associated to the primitive of $\vG$ is $\lambda$-continuous.
\end{description}

Then $\vG$ can be decomposed as the sum $\vG=f+G$, where $f$ is any strongly measurable selection of $\vG$ and $G$ is a variationally Henstock and Birkhoff integrable multifunction.
\end{thm}
\begin{proof}
Let $f$ be any Bochner measurable selection of $\vG$. Then $f$ is $HKP$-integrable and the mapping $G$ defined by
$ \vG = G + f$
is Pettis integrable (see \cite[Theorem 1]{dm2}).
Moreover, as the difference of $i(\vG)$ and $i(\{f\})$, also $i(G)$ is strongly measurable, together with $G$.\\
The variational measures $V_{\Phi}$ associated to the vH-integral
of $\vG$ and
$V_{\phi}$ associated to any $HKP$-integrable selection $f$ of $\vG$ are absolutely continuous with respect to
$\lambda$.
We can deduce then that the variational measure associated to the integral $J$ of $G$ is also $\lambda$-continuous.
So, the mapping $G$ is variationally Henstock integrable thanks to Theorem \ref{ex41}.\\  Birkhoff integrability of $G$ follows from \cite[Proposition 4.1]{cdpms2016} since the support functionals of $G$ are non-negative. \qed
\end{proof}

\begin{cor}
Let  $\vG:[0,1]\to cwk(X)$  be any variationally Henstock integrable multifunction. Then every strongly measurable selection $f$ of $\vG$ is variationally Henstock integrable and $\vG=G+f$, where $G$  is Birkhoff and variationally Henstock integrable.
\end{cor}
\begin{proof}
Since $\vG$ is variationally Henstock integrable, it satisfies (\ref{BOH}.1), (\ref{BOH}.2) and (\ref{BOH}.3), so from  Theorem \ref{BOH} we deduce that, for every strongly measurable selection $f$ the mapping $G$ defined by $\vG=G+f$ is variationally Henstock and Birkhoff integrable. Finally, since $i(\{f\})=i(\vG)-i(G)$, it follows easily that $i(\{f\})$ (and therefore also $f$) is variationally Henstock integrable.
\end{proof}

We observe that
 the   conditions (\ref{BOH}.1), (\ref{BOH}.2), (\ref{BOH}.3) do not  imply in general  variational Henstock integrability of $\vG$. But the implication holds true, as already observed, if (\ref{BOH}.2) is replaced by the stronger request that $\vG$ is Pettis integrable, by Theorem \ref{ex41}.\\

So, we shall now give an example of a mapping $f:[0,1]\to c_0$, satisfying (\ref{BOH}.1), (\ref{BOH}.2), (\ref{BOH}.3), but not Henstock integrable (and so, a fortiori, not variationally Henstock integrable).
The example is the same as \cite[Example 2]{glasgow}, and \cite[Example]{gm}.
\begin{ex}\label{exc0} \rm
Fix any disjoint sequence of closed subintervals $(J_n)_n:=([a_n,b_n])_n$ in $[0,1]$, such that $0=a_1<b_1<a_2<b_2<...$ and $\lim_n b_n=1$.

Now the function $f$ is defined as follows:
\[f_n(t):=\frac{1}2|J_{2n-1}|\chi_{J_{2n-1}}(t)-\frac{1}2|J_{2n}|\chi_{J_{2n}}(t),\]
for each $t\in [0,1]$, where $f_n$ denotes the $n^{th}$ component of $f$.
Proceeding as in the quoted papers \cite{glasgow,gm}, it is possible to see that $f$ is (HKP)-integrable, and, for every interval $I\subset [0,1]$:
\[\int_I f dt=\left(\frac{\lambda(I\cap J_{2n-1})}{2|J_{2n-1}|}-\frac{\lambda(I\cap J_{2n})}{2|J_{2n}|}\right)_n.\]
However, as observed in \cite[Example 2]{glasgow}, the set $\{\int_I f(t)dt: I\in \mathcal{I}\}$ is not relatively norm compact in $c_0$, hence $f$ is not Henstock integrable, thanks to \cite[Proposition 1]{glasgow}.

Since $c_0$ is separable, $f$ is strongly measurable; then, we only have to prove that the variational measure associated with the integral $\vG$ of $f$ is $\lambda$-continuous.

To this aim, let us fix any null set $A\in \mathcal{L}$, and fix any $\varepsilon>0$. Then clearly $\lambda(A\cap J_n)=0$ for every integer $n$. So, for every $n$ there exists an open set $U_n$ such that $A\cap J_n\subset U_n$ and $\lambda(U_n)\leq \dfrac{\varepsilon}{2^n}|J_n|$.

Now, for each element $t\in A$, we define $\delta(t)$ in such a way that $[t-\delta(t),t+\delta(t)]\subset U_n$ whenever $t\in A\cap J_n^0$, and $[t-\delta(t),t+\delta(t)]\subset ]b_{n(t)},a_{n(t)+1}[$ as soon as $t\notin \cup_n J_n$ and $n(t)$ is such that $t\in ]b_{n(t)},a_{n(t)+1}[$. Finally, if $t\in A$ is one of the points $a_n$ or $b_n$, we choose $\delta(t)$ in such a way that $[t-\delta(t),t+\delta(t)]\subset U_n$ and $[t-\delta(t),t+\delta(t)]$ intersects just one of the intervals $J_n$.

So, if $\{(I_k,t_k), k=1,...,K\}$ is any $\delta$-fine Henstock partition in $[0,1]$, with tags in $A$, for each index $k$ the interval $I_k$ intersects at most one of the intervals $J_n$, and therefore $F(I_k):=\int_{I_k}f(t)dt$ has no more than one component different from $0$ (say $n$), and
\[
\left\| \int_{I_k}f(t)dt \right\|\leq \dfrac{\lambda(I_k\cap J_n)}{2|J_n|}.\]
Hence, summing as $k=1,...,K$, we get
\begin{eqnarray*}
\sum_k\|F(I_k)\| &=&\sum_n\sum_{t_k\in J_n} \left\|\int_{I_k}f(t)dt \right\|
\leq \sum_n\sum_{t_k\in J_n}\frac{\lambda(I_k\cap J_n)}{2|J_n|}\leq \\ &\leq&
\sum_n\frac{\lambda(U_n)}{2|J_n|}\leq \sum_n\frac{\varepsilon}{2^n}=\varepsilon.
\end{eqnarray*}
This proves that $f$ satisfies condition (\ref{BOH}.3) as required.
\end{ex}

The function from the above example has a trivial decomposition: $f(t)=\{0\}+f(t)$. To have an example  of a multifunction $\vG:[0,1]\to cwk(c_0)$ satisfying  (\ref{BOH}.1),  (\ref{BOH}.2),  (\ref{BOH}.3) and possessing a non-trivial decomposition, it is sufficient to take any variationally Henstock integrable multifunction $\vG_0:[0,1]\to cwk(c_0)$, and define: $\vG(t)=\vG_0(t)+f(t)$, for $t\in [0,1]$, where $f$ is the function defined above. Since $i(\vG)=i(\vG_0)+i(\{f(t)\})$, we easily see that $i(\vG)$ (and so $\vG$) satisfies (\ref{BOH}.1),  (\ref{BOH}.2),  (\ref{BOH}.3), but it cannot be Henstock integrable, otherwise $i(\{f\})$ (and therefore $f$) would be, by difference.\\

Now, we shall give a characterization of the Birkhoff integrability  for strongly measurable integrands.
\begin{prop}\label{trecinque}
Assume that $f:[0,1]\to X $  is strongly measurable.
Then for every $\varepsilon>0$ there exists a countable partition of $[0,1]$, made of  measurable sets $(A_h)_h$, such that $f|_{A_h}$ is Bochner integrable for every $h\in\N$ and
\begin{eqnarray}\label{flaVB}
 \sum_k \left\|f(t_k)\lambda(A'_k)-\int_{A'_k}f \right\|\leq \varepsilon \end{eqnarray}
holds true, for any partition $\{A'_k\}_k$ refining  $(A_n)_n$, and any choice of  points $t_k$ in $A'_k$, $k\in \N$.
\end{prop}
\begin{proof}
Since $f$  is Bochner measurable, then $f$ is also Lusin measurable, see  \cite[Section 3]{nara}.
Now, in order to prove (\ref{flaVB}), fix arbitrarily $\varepsilon>0$.
Thanks to Lusin measurability of $f$
there exists a sequence $(B_n)_n$ of pairwise disjoint closed subsets of $[0,1]$, such that
$\sum_{n=1}^{+\infty}\lambda(B_n)=1$, and a corresponding sequence $(\delta_n)_n$ of positive real numbers, such that
$$\|f(t)-f(t')\|\leq \varepsilon$$
holds for every $n$, as soon as  $t,t'$ are in $B_n$ and $|t-t'|\leq \delta_n$. Since $f_{|B_n}$ is continuous, it is obviously Bochner integrable in $B_n$, hence we can deduce also that
\begin{eqnarray}\label{mediavera}
\left\| f(t)-\frac{\int_{B_n\cap I}f}{\lambda(B_n\cap I)}\right\|\leq \varepsilon
\end{eqnarray}
holds true, for every $n$, every interval $I\subset [0,1]$ with $|I|\leq \delta_n$ and every point $t\in B_n\cap I$ (provided the latter is non-negligible).\\
Now, in order to construct the announced partition $(A_n)_n$,
for every integer $n$ fix a finite partition $\{I_1^n,...,I^n_{N(n)}\}$ of $[0,1]$ consisting of half-open intervals of the same length, with $|I^n_j|<\delta_n$ for all $j$, and set
$A^{n+1}_j:=B_n\cap I^n_j$ whenever the intersection is non-negligible.
From (\ref{mediavera}), we see that
\begin{eqnarray}\label{quasifin}
 \left\|f(t_j^{n+1})-\dfrac{\int_{A^{n+1}_j}f}{\lambda(A^{n+1}_j)}\right\|\leq \varepsilon \end{eqnarray}
holds true, for every $n$ and every choice of points $t_j^{n+1}\in A^{n+1}_j$.
Finally, since  the sequence
$(A^{n+1}_j)_{ \{ n \in \mathbb{N}, j \in \{1, \ldots, N(n) \}\} }$ is composed by  pairwise disjoint sets,
we can rearrange them in a sequence $(A_h)_{h=2}^{\infty}$.
Adding also the null set $A_1:=[0,1]\setminus (\cup_{h=2}^{\infty} A_h)$, we claim that the partition $(A_h)_h$ obtained in this way is the requested one.\\ Indeed, let $\{A'_k\}_k$ be any refinement of $(A_h)_h$. Since each non-negligible $A'_k$ is contained in some set of the type $B_n\cap I_j$, we also have, similarly to (\ref{mediavera})
\begin{eqnarray}\label{lastref}
\left\|f(t_k)-\dfrac{\int_{A'_k}f}{\lambda(A'_k)} \right\|\leq \ve
\end{eqnarray}
as soon as $t_k\in A'_k$. So
\begin{eqnarray*}
&& \sum_k \left\|f(t_k)\lambda(A'_k)-\int_{A'_k}f \right\|\leq
\sum_n\sum_j \sum_{A'_k\subset B_n\cap I_j^n} \left\|f(t_k)\lambda(A'_k)-\int_{A'_k}f \right\|\leq \\
&& \sum_n\sum_j\sum_{A'_k\subset B_n\cap I_j^n}\ve\lambda(A'_k)= \sum_n\sum_j \ve\lambda(B_n\cap I_j^n)=\sum_n \ve\lambda(B_n)= \ve,\end{eqnarray*}
as requested.
\end{proof}
\begin{rem} \rm
We observe that  if $f$ satisfies (\ref{flaVB}),  in general $f$ is not variationally Henstock integrable: indeed, in \cite[Remark 4.3]{BDpM2} it is shown that for every Banach space $X$ there are Pettis integrable mappings $f:[0,1]\to X$ whose variational measure associated to the integral fails to be $\lambda$-continuous. In particular, if $X$ is separable, such functions are also Birkhoff integrable and strongly measurable.\\
\end{rem}

\begin{cor}
Assume that $f:[0,1]\to X$ is strongly measurable. Then $f$ is Birkhoff integrable if and only if
 there exists a countable partition $(A_k)_k$ of $[0,1]$ such that the restriction $f |_{A_k}$ is bounded whenever $\lambda(A_k) > 0$ and
$\left\{ \sum_{k} f(t_k) \lambda(A_k), t_k \in A_k \right\}$ is made up of  unconditionally convergent series.
\end{cor}
\begin{proof}
It is enough to apply Proposition 3.5 together with \cite[Proposition 2.2]{crbb}.
\end{proof}

By the properties of the   R{\aa}dstr\"{o}m embedding we can obtain also  that a Bochner measurable multifunction $\vG: [0,1] \to cwk(X)$ is  Birkhoff integrable if and only
 is there exists a countable partition of $[0,1]$: $(A_k)_k$ such that the restriction $F |_{A_k}$ is bounded whenever $\lambda(A_k) > 0$ and
$\left\{ \sum_{k}\vG(t_k) \lambda(A_k), t_k \in A_k \right\}$ is made up of  unconditionally convergent series.
\\

Finally, we would like to remark that, for strongly measurable maps, Birkhoff integrability can be labeled in a form that recalls the notion of variational Henstock integrability (though there is no direct implication between the two notions, in general).
We first give a definition.
\begin{deff}\label{sBi}\rm
Given a function $f:[0,1]\to X$, we say that $f$ is {\em strongly Birkhoff} integrable if there exists a $\lambda$-continuous $\sigma$-additive measure $\phi:\mathcal{L}\to X$ such that, for every $\varepsilon>0$ it is possible to find a countable partition $P$ of $[0,1]$, made with measurable sets, such that, as soon as $\{A_n\}_n$ is a refinement of $P$, it holds
\[\sum_n \|f(t_n)\lambda(A_n)-\phi(A_n)\|\leq \ve\]
for every choice of $t_n\in A_n$, $n\in \mathbb{N}$.
\end{deff}

\noindent We now observe that, as soon as $f:[0,1]\to X$ is strongly measurable and Birkhoff integrable, it turns out to be also {\em strongly Birkhoff} integrable, thanks to Proposition \ref{trecinque}, and of course $\phi$ is the (Birkhoff) integral function of $f$. 
\\

Conversely, we have the following results.
\begin{prop}
Assume that $f:[0,1]\to X$ is strongly Birkhoff integrable. Then it is also Birkhoff integrable.
\end{prop}
\begin{proof}
Fix $\ve>0$, and let $P:=\{E_j:j\in \mathbb{N}\}$ be the corresponding partition in the definition of strong Birkhoff integrability. Choose now any finer partition $\{A_k\}_k$,
set $x:=\sum_k\phi(E_k)=\phi([0,1])$, and observe that there exists an integer $N$ such that $\|x-\sum_{k\leq n}\phi(A_k)\|\leq \ve$ for every $n\geq N$. So, for every $n\geq N$, and every choice of points $t_k\in A_k$, we have
\begin{eqnarray*}
&& \|\sum_{k=1}^nf(t_k)\lambda(A_k)-x\|\leq
\|\sum_{k=1}^nf(t_k)\lambda(A_k)-\sum_{k=1}^n \phi(A_k)\|+\ve\leq \\
&&\leq \sum_{k=1}^n\|f(t_k)\lambda(A_k)-\phi(A_k)\|+\ve\leq
 \sum_{k=1}^{\infty}\|f(t_k)\lambda(A_k)-\phi(A_k)\| +\ve \leq 2\ve.
\end{eqnarray*}
This means that
\[\limsup_n \|\sum_{k=1}^nf(t_k)\lambda(A_k)-x\|\leq 2\ve.\]
This property in \cite{ccgs} is called {\em simple-Birkhoff} integrability, and is proved in \cite[Theorem 3.18]{ccgs} to be equivalent to Birkhoff integrability of $f$.
\end{proof}
\begin{prop}
Assume that $f:[0,1]\to X$ is strongly Birkhoff integrable. Then $f$ is also strongly measurable.
\end{prop}
\begin{proof}
We shall follow the technique of \cite[Lemma 3]{dm11}. First of all, let us denote by $H$ the separable subspace of $X$ generated by the range of $\phi$. Next, for every $\ve>0$ let us set
\[T_{\ve}:=\{t\in [0,1]: d(f(t),H)>\ve\}.\]
Now, we shall prove that $\lambda(T_{\ve})=0$ for all $\ve>0$, from which it will follow that $f$ is separably-valued, and so also strongly measurable, thanks to the Pettis Theorem.
So, assume by contradiction that there exists a $\ve'>0$ such that $\lambda^*(T_{\ve'})=a>0$. Let $A'$ denote any measurable set such that  $T_{\ve'}\subset A'$ and $\lambda(A')=\lambda^*(T_{\ve'})=a$.
\\
Now, fix $\ve<\ve'$ and take any partition $\Pi\equiv \{E_j\}_j$ satisfying the condition of strong Birkhoff integrability of $f$ with respect to $a\ve$. Moreover let $\{A_n\}_n$ be any finer partition, such that $A'\cap E_j$ and $E_j\setminus A'$ are among the sets $A_n$. Without loss of generality, we can assume that $\lambda(A_n)>0$ for all $n$.
Observe that, if $A_n$ is of the type $A'\cap E_j$, then $A_n\cap T_{\ve'}$ is non-empty, otherwise $A'\setminus A_n$ would contain $T_{\ve'}$ but would have smaller measure than $a$. So, we can choose $t_n\in A_n\cap  T_{\ve'}$ as soon as $A_n$ is of the type $A'\cap E_j$. Thus we have
\[\sum'_n\|f(t_n)\lambda(A_n)-\phi(A_n)\|\leq a \ve\]
where the $\sum'$ runs along the sets $A_n$ of the type $A'\cap E_j$.
Therefore, since $A'$ is the union of all the sets $A_n$ of the type above,
\[\sum'_n\|f(t_n)\lambda(A_n)-\phi(A_n)\|\leq \ve \sum'_n\lambda(A_n)\]
i.e.
\[\sum'_n\left\| f(t_n)-\dfrac{\phi(A_n)}{\lambda(A_n)} \right\|\lambda(A_n)\leq \ve \sum'_n\lambda(A_n).\]
This inequality implies that, for at least one integer $n$, it is
\[\left\| f(t_n)-\dfrac{\phi(A_n)}{\lambda(A_n)} \right\|\leq \ve<\ve',\]
but this is impossible, since $\dfrac{\phi(A_n)}{\lambda(A_n)}\in H$ and $t_n\in T_{\ve'}$. This proves that $\lambda(T_{\ve})=0$ for every $\ve>0$, and so $f$ is strongly measurable.
\end{proof}
Summarizing, we get the following conclusion.

\begin{thm}
Assume that $f:[0,1]\to X$ is any mapping. Then $f$ is strongly Birkhoff integrable if and only if it is strongly measurable and Birkhoff integrable.
\end{thm}

\section*{Acknowledgements}
This is a post-peer-review, pre-copyedit version of an article published in Ricerche di Matematica. The final authenticated version is available online at: http://dx.doi.org/10.1007/s11587-018-0376-x.
\small

\end{document}